\documentclass[12pt, reqno]{amsart}
\usepackage{amsmath, amsthm, amscd, amsfonts, amssymb, graphicx, color}
\usepackage[bookmarksnumbered, colorlinks, plainpages]{hyperref}
\usepackage{verbatim}

\newtheorem{theorem}{Theorem}[section]

\newtheorem{corollary}[theorem]{Corollary}
\theoremstyle{definition}

\theoremstyle{remark}
\newtheorem{remark}[theorem]{Remark}
\numberwithin{equation}{section}

\newcommand{\w}{\widetilde}
\newcommand{\N}{{\mathbb{N}}}
\newcommand{\Z}{{\mathbb{Z}}}
\renewcommand{\O}{{\mathcal{O}}}
\newcommand{\cstar}{\mbox{$C^*$}}

\newcommand{\cspan}{\mathop{\overline{\textrm{span}}}}
\newcommand{\p}[2]{\langle #1,#2\rangle}
\newcommand{\Sum}{\mathop{\mbox{$\sum$}}}

\renewcommand{\mod}{\mathop{\textrm{mod}}}
\renewcommand{\Im}{\mathop{\textrm{Im}}}

\begin{document}
\setcounter{page}{1}

\title[Green's theorem for crossed products by Hilbert $\cstar$-bimodules]{Green's theorem for crossed products by Hilbert $\cstar$-bimodules}
\author[M. Achigar]{Mauricio Achigar}
\address{Centro Universitario Litoral Norte, Gral. Rivera 1350, CP 50000, Salto, Uruguay.}
\email{\textcolor[rgb]{0.00,0.00,0.84}{mauricio.achigar3@gmail.com}}

\date{Received: xxxxxx; Revised: yyyyyy; Accepted: zzzzzz. \newline\indent  Partially supported by Proyecto Fondo Clemente Estable FCE2007\!\_731.
}

\begin{abstract}
Green's theorem gives a Morita equivalence $C_0(G/H,A)\rtimes G\sim A\rtimes H$ for a closed subgroup $H$ of a locally compact group $G$ acting on a $C^*$-algebra $A$. We prove an analogue of Green's theorem in the case $G=\Z$, where the automorphism generating the action is replaced by a Hilbert $C^*$-bimodule. 
\end{abstract} \maketitle

\section{Introduction}

The crossed product $A\rtimes X$ of a $C^*$-algebra $A$ by a Hilbert $A-A$ bimodule $X$, as defined in \cite{AEE}, is a generalization of the crossed product $A\rtimes_{\alpha}\Z$ of $A$ by an automorphism $\alpha$ of $A$. 
Given an automorphism $\alpha$ of $A$ one can twist the trivial bimodule $_AA_A$ replacing the right structure by defining $x\cdot_{\alpha}a=x\alpha(a)$ and $\langle x,y\rangle^{\alpha}_R=\alpha^{-1}(a^*b)$ for $a,x,y\in A$, to get a $C^*$-bimodule, denoted by $A_{\alpha}$, satisfying $A\rtimes_{\alpha}\Z\cong A\rtimes A_{\alpha}$ canonically.

Green's theorem, as stated in \cite[Theorem 4.22]{W}, gives a Morita equivalence $C_0(G/H,A)\rtimes G\sim A\rtimes_{\alpha|_H}H$ for a general locally compact $C^*$-dynamical system $(A,G,\alpha)$ and a closed subgroup $H\leq G$. In the special case $G=\Z$, $H=n\Z$, for $n\in\N$, we have $G/H=\Z_n$ so that $C_0(G/H,A)=C_0(\Z_n,A)\cong A^n$ ($n$-fold direct sum) and $(A,H,\alpha|_H)=(A,n\Z,\alpha|_{n\Z})\cong(A,\Z,\alpha^n)$ so that $A\rtimes_{\alpha|_H}H\cong A\rtimes_{\alpha^n}\Z$, where $\alpha$ also denotes the single automorphism generating the action of $\Z$ on $A$, and $\alpha^n$ its $n$-th composition power. Then, for this special case, we have the Morita equivalence $A^n\rtimes_{\sigma}\Z\sim A\rtimes_{\alpha^n}\Z$ for a certain action $\sigma$ on $A^n$. Translating this into the $C^*$-bimodule language we get $A^n\rtimes A^n_{\sigma}\sim A\rtimes A_{\alpha^n}\cong A\rtimes [A_{\alpha}]^{\otimes n}$, where we use the isomorphism $A_{\alpha^n}\cong [A_{\alpha}]^{\otimes n}$ ($n$-fold tensor product).

In this context, we show that one can replace $A_{\alpha}$ by a general right full Hilbert $A-A$ bimodule $X$ and establish a Morita equivalence of the form 
  $$A^n\rtimes X^n_{\sigma}\sim A\rtimes X^{\otimes n}.$$ 

We obtain this as a consequence of Theorem \ref{elteo}, which states a Morita equivalence of the form
  $$(A_1\oplus\cdots\oplus A_n)\rtimes (X_1\oplus\cdots\oplus X_n)_{\sigma}\sim A_1\rtimes(X_1\otimes\cdots\otimes X_n),$$
for a ``cycle'' of bimodules $_{A_1}{X_1}_{A_2},{_{A_2}}{X_2}_{A_3},\ldots,{_{A_{n-1}}}{X_{n-1}}_{A_n},{_{A_n}}{X_n}_{A_1}$, the especial case $A_i=A$, $X_i=X$, $i=1,\ldots,n$, giving the desired result.

\section{Preliminaries}

\subsection{\it $C^*$-modules, $C^*$-bimodules, equivalence bimodules and fullness.} 
A \emph{right Hilbert $B$-module} $X_B$ is defined as a vector space $X$ equipped with a right action of the $C^*$-algebra $B$ and a $B$-valued right inner product, which is complete with respect to the induced norm. A left Hilbert $A$-module $_AX$ is defined analogously. A \emph{Hilbert $A-B$ bimodule} $_AX_B$ is a vector space $X$ with left and right compatible Hilbert $C^*$-module structures over $C^*$-algebras $A$ and $B$, respectively. Compatibility means that $\langle x,y\rangle _L\cdot z=x\cdot\langle y,z\rangle_R$, for all $x,y,z\in X$.
We say that a Hilbert $A-B$ bimodule is \emph{right full} if $\langle X,X\rangle_R=B$, where $\langle X,X\rangle_R=\cspan(\{\langle x,y\rangle_R:x,y\in X\})$, $\cspan$ denoting the closed linear spanned set. Left fullness is defined analogously. Finally, an \emph{equivalence bimodule} is a Hilbert $A-B$ bimodule $_AX_B$ which is right full and left full. When an equivalence bimodule $_AX_B$ exists the $C^*$-algebras $A$ and $B$ are said to be \emph{Morita equivalent}, a situation denoted $A\sim B$. See \cite{La} for reference.

\subsection{\it Operations with subspaces.} For linear subspaces $X, X_1,\ldots,X_n$ of a fixed normed $*$-algebra $C$ we define
\begin{center}
  $\Sum X_i\equiv X_1+X_2+\cdots +X_n\equiv\overline{\bigl\{ x_1+x_2+\cdots+x_n:x_i\in X_i\bigr\}},$\\ \vspace{1mm}
  $\textstyle \prod X_i\equiv X_1X_2\cdots X_n\equiv\overline{\bigl\{\Sum_kx_{1k}x_{2k}\cdots x_{nk}:x_{ik}\in X_i\bigr\}}, \quad X^*\equiv\bigl\{x^*:x\in X\bigr\}.$
\end{center}
If $Y_1,\ldots,Y_n$ is another family of subspaces and $\overline{X_i}=\overline{Y_i}$ for $i=1,\ldots,n$ then $\sum X_i=\sum Y_i$ and $\prod X_i=\prod Y_i$. Consequently, equalities of the form $XY=\overline{X}Y$, $X+Y=\overline{X}+Y$, etc. hold for subspaces $X$ and $Y$. Also, the following  properties are easily checked for subspaces $X,Y,Z$.
  \begin{center}
  \begin{minipage}{.85\linewidth}
    1. $(X+Y)+Z=X+Y+Z=X+(Y+Z)$,
    \hfill2. $X+Y=Y+X$,\\
    3. $(XY)Z=XYZ=X(YZ)$,
    \hfill4. $X(Y+Z)=XY+XZ$,\\
    5. $(X+Y)^*=X^*+Y^*$,\hfill6. $(XY)^*=Y^*X^*$,
    \hfill7. $(X^*)^*=X$.
  \end{minipage}
  \end{center}
For a general family of subspaces $\{X_i\}_{i\in I}$ we extend the definition of sum as
  $$\textstyle\sum X_i\equiv\overline{\bigl\{\sum_{i\in I_0}x_i:I_0\subseteq I\text{ finite},x_i\in X_i\bigr\}}.$$
For every such a family and a subspace $X$ we have
\begin{center}
 8. $X(\sum X_i)=\sum XX_i$,\qquad9. $\bigl(\sum X_i\bigr)^*=\sum X_i^*$.
\end{center}

\subsubsection{}\label{power}
Let $C$ be a fixed normed $*$-algebra, $A\subseteq C$  a $*$-subalgebra and $X\subseteq C$ a linear subspace such that
  $$\text{1. }AX\subseteq X,\quad\text{2. }XA\subseteq X,\quad\text{3. }X^*X\subseteq A,\quad\text{4. }XX^*\subseteq A.$$
For $k\in\Z$, we define $X^k=XX\cdots X$ ($k$ times) if $k\geq1$, $X^0=A$ and $X^k=(X^*)^{-k}$ if $k\leq-1$. We have $X^kX^l\subseteq X^{k+l}$ for all $k,l\in\Z$, and $X^kX^l=X^{k+l}$ if $kl>0$. Denote with $A[X]$ the closed $*$-subalgebra of $C$ generated by $A\cup X$. That is
  $$\textstyle A[X]=C^*(A\cup X)=\sum_{k\in\Z}X^k.$$

\subsubsection{}\label{conmut} With $A,X\subseteq C$ as before, let $B\subseteq C$ be a $*$-subalgebra. Note that 
\begin{center}
 {\it if $BA=AB$ and $BX=XB$ then $BA[X]=A[X]B$.}
\end{center}
  Indeed, in this case $BX^k=X^kB$ for all $k\in\Z$ and then 
  $$\textstyle BA[X]=B\Sum_kX^k=\Sum_kBX^k=\Sum_kX^kB=A[X]B.$$ 
  In a similar fashion, we can prove that
\begin{center}
  {\it if $BA=A$ and $BX=X$ then $BA[X]=A[X]$.}
\end{center}

\subsubsection{}\label{power2} If in the context of \ref{power} $C$ is a $C^*$-algebra and $A$ and $X$ are closed, then $A$ is a $C^*$-algebra and $X$ a Hilbert $A-A$ bimodule with the operations given by the restriction of the trivial Hilbert $C-C$ bimodule structure of $C$. Then we have $AX=X$ and $XA=X$ because both actions are automatically non-degenerate. Moreover, if we assume that $X$ is right full, that is $X^*X=A$, then we have $X^{-k}X^l=X^{l-k}$ for $k,l\geq0$.

\subsection{\it Crossed product by a Hilbert bimodule.} Crossed products of $C^*$-algebras by Hilbert bimodules are introduced in \cite{AEE}. We summarize here their definition and principal properties. 

\subsubsection{Covariant pairs.} Given a Hilbert $A-A$ bimodule $X$ and a $C^*$-algebra $C$ a {\emph covariant pair} from $_AX_A$ to $C$ is a pair of maps $(\varphi,\psi)$ where $\varphi\colon A\to C$ is a $*$-morphism and $\psi\colon X\to C$ a linear map satisfying
\begin{center}
 $\text{1. }\psi(a\cdot x)=\varphi(a)\psi(x),\quad\text{2. }\varphi(\langle x,y\rangle_L)=\psi(x)\psi(y)^*,$\\
 $\text{3. }\psi(x\cdot a)=\psi(x)\varphi(a),\quad\text{4. }\varphi(\langle x,y\rangle_R)=\psi(x)^*\psi(y),$
\end{center}
for all $a\in A$, $x,y\in X$. That is, the pair preserves the Hilbert bimodule structure considering on $C$ the trivial Hilbert $C-C$ bimodule structure.

\subsubsection{The crossed product.}\label{univprop} A crossed product of a $C^*$-algebra $A$ by a Hilbert $A-A$ bimodule $X$ is a $C^*$-algebra $A\rtimes X$ (denoted $A\rtimes_X\Z$ in \cite{AEE}) together with a covariant pair $(\iota_A,\iota_X)$ from $_AX_A$ to $A\rtimes X$ satisfying the following universal property: {\it for any covariant pair $(\varphi,\psi)$ from $_AX_A$ to a $C^*$-algebra $C$ there exists a unique $*$-morphism $\varphi\rtimes\psi\colon A\rtimes X\to C$ such that $\varphi=(\varphi\rtimes\psi)\circ\iota_A$ and $\psi=(\varphi\rtimes\psi)\circ\iota_X$.}

\subsubsection{Basic properties.}\label{CPprops} The crossed product exists and is unique up to isomorphism. The maps $\iota_A$ and $\iota_X$ are injective, so that we may consider $A,X\subseteq A\rtimes X$ and the induced $*$-morphism $\varphi\rtimes\psi$ as an extension of the covariant pair $(\varphi,\psi)$. Moreover, for any covariant pair $(\varphi,\psi)$ we have that $\Im\varphi\rtimes\psi=C^*(\Im\varphi\cup\Im\psi)$ and that $\varphi\rtimes\psi$ is injective if $\varphi$ is.

\section{The main theorem.}

\subsection{\it Twisting Hilbert modules.} If $X_B$ is a right Hilbert $B$-module, $C$ a $C^*$-algebra and $\sigma\colon C\to B$ a $*$-isomorphism, then we denote $X_{\sigma}$ the right Hilbert module over $C$ obtained by considering on the vector space $X$ the operations
$$x\cdot_{\sigma}c=x\cdot\sigma(c)\quad\text{and}\quad\p xy^{\sigma}=\sigma^{-1}(\p xy)\qquad\text{for }c\in C,\,x,y\in X.$$
If in addition $X$ is a Hilbert $A-B$ bimodule then $X_{\sigma}$ is a Hilbert $A-C$ bimodule with the original left structure. The module $X$ is right full iff $X_{\sigma}$ is.

\subsection{\it The twisted sum of a cycle of Hilbert bimodules.} \label{twisted} Given Hilbert bimodules $_{A_i}{X_i}_{B_i}$ for $i=1,\ldots,n$, we have that $\bigoplus X_i$ is a Hilbert $\bigoplus A_i-\bigoplus B_i$ bimodule with point-wise operations. The bimodule $\bigoplus X_i$ is right full iff $X_i$ is for all $i=1,\ldots,n$.

Now, given a ``cycle'' of Hilbert bimodules $_{A_1}{X_1}_{A_2}$, $_{A_2}{X_2}_{A_3}$, $\ldots$, $_{A_n}{X_n}_{A_1}$ we can make $\bigoplus X_i$ into a Hilbert bimodule over $\bigoplus A_i$ twisting the right action in the previous constriction with the isomorphism $\sigma\colon A_1\oplus A_2\oplus\cdots\oplus A_n\to A_2\oplus\cdots\oplus A_n\oplus A_1$ given by
  $$\sigma(a_1,a_2,\ldots,a_n)=(a_2,\ldots,a_n,a_1),\qquad\text{for }a_k\in A_k.$$ 

\begin{theorem}\label{elteo}
Let $_{A_1}{X_1}_{A_2}$, $_{A_2}{X_2}_{A_3}$, $\ldots$, $_{A_n}{X_n}_{A_1}$ be right full Hilbert bimodules and consider their twisted sum $(X_1\otimes\cdots\otimes X_n)_{\sigma}$ as in \ref{twisted}. Then we have the following Morita equivalence
  $$A_1\rtimes(X_1\otimes\cdots\otimes X_n)\sim(A_1\oplus\cdots\oplus A_n)\rtimes (X_1\oplus\cdots\oplus X_n)_{\sigma}.$$
\end{theorem}

\begin{proof}
Let $A=A_1\oplus\cdots\oplus A_n$, $X=(X_1\oplus\cdots\oplus X_n)_{\sigma}$ and $C=A\rtimes X$. We may suppose that $A_k\subseteq A\subseteq C$ and $X_k\subseteq X\subseteq C$, for $k=1,\ldots,n$, so that the module operations of each bimodule $X_k$ and also the ones of the bimodule $X$ are given by the operations of the $C^*$-algebra $C$, i.e., by the restriction of the trivial Hilbert $C-C$ bimodule structure of $C$. Note that the spaces $A,X\subseteq C$ verify the conditions in \ref{power}, then we can define $X^k$ for $k\in\Z$ as done there. Moreover, as $X$ is a Hilbert $A-A$ bimodule (hence, non-degenerate for both actions) and is right full, because each $X_k$ is, we have that $A,X\subseteq C$ verify the conditions of \ref{power2}.

We extend the families $\{A_k\}_{k=1}^n$ and $\{X_k\}_{k=1}^n$ to families $\{A_k\}_{k\in\Z}$ and $\{X_k\}_{k\in\Z}$ letting $A_k=A_l$ and $X_k=X_l$ if $k=l\mod n$. For all $k\in\Z$ we have 
  $$A_kX_k=X_k=X_kA_{k+1},\quad X_kX_k^*\subseteq A_k,\quad\text{and}\quad X_k^*X_k=A_{k+1},$$
because each $X_k$ is a Hilbert $A_k-A_{k+1}$ bimodule (hence non-degenerate for both actions) and right full.
We also have 
$$A_kA_l=A_kX_l=X_kA_{l+1}=0\quad\text{for }k,l\in\Z,\quad k\neq l\mod n,$$
therefore, as $A=\sum_{k=1}^nA_k$ and $X=\sum_{k=1}^nX_k$,
  $$A_k=A_kA=AA_k,\qquad X_k=A_kX=XA_{k+1}\quad\text{for }k\in\Z,$$
and then 
  $$A_kX^l=X^lA_{k+l}\quad\text{for }k,l\in\Z.$$

In particular, for $k\in\Z$ we have that $A_kX^n=X^nA_k$ so that the pairs $A_k$, $A_kX^n$ satisfy the conditions of \ref{power} and \ref{power2}. Following the notation in \ref{power} we define the $C^*$-subalgebra
  $$B=A_1[A_1X^n]=C^*(A_1\cup A_1X^n)\subseteq C,$$
and the closed subspace
  $$Z=BX^0\oplus BX\oplus\cdots\oplus BX^{n-1}\subseteq M_{1\times n}(C),$$
where $M_{1\times n}(C)$ is considered as a Hilbert $C-M_n(C)$ bimodule with the usual matrix operations. To prove the theorem it is enough to show that $ZZ^*\subseteq C$ and $Z^*Z\subseteq M_n(C)$ are $C^*$-subalgebras, that $Z$ is an equivalence  $ZZ^*$-$Z^*Z$ bimodule with the restricted (matrix) operations and that we have isomorphisms $ZZ^*\cong A_1\rtimes(X_1\otimes\cdots\otimes X_n)$ and $Z^*Z\cong C$.

\par Let us consider the issues concerning the left side first. Notice that the equalities $AA_1=A_1=A_1A$ and $A(A_1X^n)=A_1X^n=(A_1X^n)A$ implies, by \ref{conmut}, that $AB=B=BA$, because $B=A_1[A_1X^n]$ by definition.
Then we can calculate
  $$ZZ^*=\Sum_{k=0}^{n-1}(BX^k)(BX^k)^*=B\Bigl(\Sum_{k=0}^{n-1}X^kX^{-k}\Bigr)B=BAB=B,$$
where we use that $\sum_{k=0}^{n-1}X^kX^{-k}=A$, which is a consequence of the relations $X^kX^{-k}\subseteq A$ and $X^0=A$, and that $AB=B$. Besides, form the definition of $Z$ it is apparent that $Z$ is $B$-invariant for the left action. Then we have shown that $Z$ is a full left Hilbert $B$-module. Finally, to see that $B\cong A_1\rtimes(X_1\otimes\cdots\otimes X_n)$, consider the covariant pair $(\varphi, \psi)$ given by
  $$\varphi\colon A_1\to C,\qquad \varphi(a)=a\quad\text{ for }a\in A_1,$$
  $$\psi\colon X_1\otimes\cdots\otimes X_n\to C,\qquad\psi(x_1\otimes\cdots\otimes x_n)=x_1\cdots x_n\quad\text{ for }x_k\in X_k.$$
By the universal property of the crossed product (\ref{univprop}) this covariant pair extends to a $*$-morphism $\varphi\rtimes\psi\colon A_1\rtimes(X_1\otimes\cdots\otimes X_n)\to C$, which is injective because $\varphi$ is (\ref{CPprops}). Moreover, $\Im\varphi\rtimes\psi=C^*(\Im\varphi\cup\Im\psi)=B$ because $\Im\varphi=A_1$, $\Im\psi=X_1\cdots X_n=A_1X\cdots A_nX=A_1X^n$ and $B=C^*(A_1\cup A_1X^n)$ by definition. Hence we have the desired isomorphism.
    
\par Now, turning to the right side, note that $Z^*Z$ is clearly a closed self-adjoint subspace of $M_n(C)$. Form the equalities $ZZ^*=B$ and $BZ=Z$ we also deduce that $(Z^*Z)(Z^*Z)=Z^*BZ=Z^*Z$, so that $Z^*Z$ is a $C^*$-subalgebra of $M_n(C)$ and $Z$ is a full right Hilbert $Z^*Z$-module.

To show that $Z^*Z\cong C$, consider the pair of maps $\varphi\colon A\to M_n(C)$ and $\psi\colon X\to M_n(C)$ given by
\arraycolsep=2.5pt 
\medmuskip = 1mu 
  $$\varphi(a_1,\ldots,a_n)=
  \left[\begin{array}{cccc}
  a_1     &0      &\cdots     &0      \\
  0       &a_2    &\ddots     &\vdots \\
  \vdots  &\ddots &\ddots     &0      \\
  0       &\cdots &0          &a_n
  \end{array}\right],\quad
  \psi(x_1,\ldots,x_n)=
  \left[\begin{array}{cccc}
  0       &x_1     &\cdots     &0      \\
  0       &0       &\ddots     &\vdots \\
  \vdots  &\ddots  &\ddots     &x_{n-1}      \\
  x_n     &\cdots  &0          &0
  \end{array}\right],$$
\medmuskip = 4mu plus 2mu minus 4mu  
for $(a_1,\ldots,a_n)\in A=A_1\oplus\cdots\oplus A_n$ and $(x_1,\ldots,x_n)\in X=(X_1\oplus\cdots\oplus X_n)_{\sigma}$.

The following calculations shows that this pair is a covariant pair. For every $a_k\in A_k$, $x_k,y_k\in X_k$, $k=1,\ldots,n$ we have
\medmuskip = 1mu
\begin{multline*}
  \psi((a_1,\ldots,a_n)\cdot(x_1,\ldots,x_n))=
  \psi(a_1\cdot x_1,\ldots,a_n\cdot x_n)\\
=\left[\scriptstyle\begin{array}{cccc}
  0       &a_1\cdot x_1     &\cdots     &0      \\
  0       &0       &\ddots     &\vdots \\
  \vdots  &\ddots  &\ddots     &a_{n-1}\cdot x_{n-1}      \\
  a_n\cdot x_n     &\cdots  &0          &0
  \end{array}\right]
=\left[\begin{array}{cccc}
  a_1     &0      &\cdots     &0      \\
  0       &a_2    &\ddots     &\vdots \\
  \vdots  &\ddots &\ddots     &0      \\
  0       &\cdots &0          &a_n
  \end{array}\right]
    \left[\begin{array}{cccc}
  0       &x_1     &\cdots     &0      \\
  0       &0       &\ddots     &\vdots \\
  \vdots  &\ddots  &\ddots     &x_{n-1}      \\
  x_n     &\cdots  &0          &0
  \end{array}\right]\\
=\varphi(a_1,\ldots,a_n)\psi(x_1,\ldots,x_n);
\end{multline*}
\begin{multline*}
  \varphi(\p{(x_1,\ldots,x_n)}{(y_1,\ldots,y_n)}_L)=
  \varphi(\p{x_1}{y_1}_L,\ldots,\p{x_n}{y_n}_L)\\
=\left[\begin{array}{cccc}
  \p{x_1}{y_1}_L     &0      &\cdots     &0      \\
  0       &\p{x_2}{y_2}_L    &\ddots     &\vdots \\
  \vdots  &\ddots &\ddots     &0      \\
  0       &\cdots &0          &\p{x_n}{y_n}_L
  \end{array}\right]
=\left[\begin{array}{cccc}
  0       &x_1     &\cdots     &0      \\
  0       &0       &\ddots     &\vdots \\
  \vdots  &\ddots  &\ddots     &x_{n-1}      \\
  x_n     &\cdots  &0          &0
  \end{array}\right]
  \left[\begin{array}{cccc}
  0       &0       &\cdots     &y_n^*  \\
  y_1^*   &0       &\ddots     &\vdots \\
  \vdots  &\ddots  &\ddots     &0      \\
  0       &\cdots  &y_{n-1}^*  &0
  \end{array}\right]\\
=\psi(x_1,\ldots,x_n)\psi(y_1,\ldots,y_n)^*;
\end{multline*}
\begin{multline*}
  \psi((x_1,\ldots,x_n)\cdot_{\sigma}(a_1,\ldots,a_n))=
  \psi((x_1,\ldots,x_n)\cdot \sigma(a_1,\ldots,a_n))\\
=\psi((x_1,\ldots,x_n)\cdot(a_2,\ldots,a_n,a_1))
=\psi(x_1\cdot a_2,\ldots,x_{n-1}\cdot a_n, x_n\cdot a_1)\\
=\left[\begin{array}{cccc}
  0       &x_1\cdot a_2     &\cdots     &0      \\
  0       &0       &\ddots     &\vdots \\
  \vdots  &\ddots  &\ddots     &x_{n-1}\cdot a_n      \\
  x_n\cdot a_1     &\cdots  &0          &0
  \end{array}\right]
=\left[\begin{array}{cccc}
  0       &x_1     &\cdots     &0      \\
  0       &0       &\ddots     &\vdots \\
  \vdots  &\ddots  &\ddots     &x_{n-1}      \\
  x_n     &\cdots  &0          &0
  \end{array}\right]
  \left[\begin{array}{cccc}
  a_1     &0      &\cdots     &0      \\
  0       &a_2    &\ddots     &\vdots \\
  \vdots  &\ddots &\ddots     &0      \\
  0       &\cdots &0          &a_n
  \end{array}\right]\\
=\psi(x_1,\ldots,x_n)\varphi(a_1,\ldots,a_n);
\end{multline*}
\arraycolsep=.5 pt 
\medmuskip = -2mu 
\begin{multline*}
  \varphi(\p{(x_1,\ldots,x_n)}{(y_1,\ldots,y_n)}_R^{\sigma})=
  \varphi(\sigma^{-1}(\p{(x_1,\ldots,x_n)}{(y_1,\ldots,y_n)}_R))\\
=\varphi(\sigma^{-1}(\p{x_1}{y_1}_R,\ldots,\p{x_n}{y_n}_R))
=\varphi(\p{x_n}{y_n}_R,\p{x_1}{y_1}_R,\ldots,\p{x_{n-1}}{y_{n-1}}_R)\\
=\left[\begin{array}{cccc}
  \p{x_n}{y_n}_R     &0      &\cdots     &0      \\
  0       &\p{x_1}{y_1}_R    &\ddots     &\vdots \\
  \vdots  &\ddots &\ddots     &0      \\
  0       &\cdots &0          &\p{x_{n-1}}{y_{n-1}}_R
  \end{array}\right]
=\left[\begin{array}{cccc}
  0       &0       &\cdots     &x_n^*  \\
  x_1^*   &0       &\ddots     &\vdots \\
  \vdots  &\ddots  &\ddots     &0      \\
  0       &\cdots  &x_{n-1}^*  &0
  \end{array}\right]
  \left[\begin{array}{cccc}
  0       &y_1     &\cdots     &0      \\
  0       &0       &\ddots     &\vdots \\
  \vdots  &\ddots  &\ddots     &y_{n-1}      \\
  y_n     &\cdots  &0          &0
  \end{array}\right]\\
=\psi(x_1,\ldots,x_n)^*\psi(y_1,\ldots,y_n).
\end{multline*}
\arraycolsep=2.5pt 
\medmuskip = 4mu plus 2mu minus 4mu  

By the universal property of the crossed product the covariant pair $(\varphi, \psi)$ extends to a $*$-morphism $\varphi\rtimes\psi\colon A\rtimes X\to M_n(C)$, which is injective because $\varphi$ is. Then, to end the proof it suffices to show that $\Im \varphi\rtimes\psi=Z^*Z$.
  
We calculate $Z^*Z\subseteq M_n(C)$ adopting the following matrix notation
  $$Z^*Z=[E_{ ij}]_{i,j=1}^n\quad\text{where}\quad E_{ij}=(BX^{i-1})^*(BX^{j-1})\subseteq C,\quad\text{for }i,j=1,\ldots,n.$$
Simplifying the expressions of the $E_{ij}$'s we get 
  $$E_{ij}=(BX^{i-1})^*(BX^{j-1})=X^{1-i}BX^{j-1},\quad\text{for }i,j=1,\ldots,n.$$
Note that 
  $$E_{ii}=X^{1-i}BX^{i-1}\supseteq X^{1-i}A_1X^{i-1}=X^{1-i}X^{i-1}A_i=A_i\qquad\text{for}\quad i=1,\ldots,n,$$
  $$E_{ii+1}=E_{ii}X\supseteq A_iX=X_i\qquad\text{for}\quad i=1,\ldots,n-1\quad\text{and}$$
  $$E_{n1}=X^{1-n}BX^0\supseteq X^{1-n}(A_1X^n)A=A_nX^{1-n}X^n=A_nX=X_n.$$
Then we see that $\Im\varphi\subseteq Z^*Z$ and $\Im\psi\subseteq Z^*Z$. As $\Im\varphi\rtimes\psi= C^*(\Im\varphi\cup\Im\psi)$ we conclude that $\Im\varphi\rtimes\psi\subseteq Z^*Z.$
  
To prove the reverse inclusion denote $D=\Im\varphi\rtimes\psi= C^*(\Im\varphi\cup\Im\psi)$,
  $$\Im\varphi=
  \left[\begin{array}{cccc}
  A_1     &0      &\cdots     &0      \\
  0       &A_2    &\ddots     &\vdots \\
  \vdots  &\ddots &\ddots     &0      \\
  0       &\cdots &0          &A_n
  \end{array}\right]\subseteq D\quad\text{and}\quad
  \w X=\Im\psi=
  \left[\begin{array}{cccc}
  0       &X_1     &\cdots     &0      \\
  0       &0       &\ddots     &\vdots \\
  \vdots  &\ddots  &\ddots     &X_{n-1}      \\
  X_n     &\cdots  &0          &0
  \end{array}\right]\subseteq D.$$
  
Note that 
$$
 \left[\begin{array}{cccc}
  0       &X_1     &\cdots     &0      \\
  0       &0       &0       &\vdots \\
  \vdots  &\ddots  &\ddots     &0      \\
  0     &\cdots  &0          &0
  \end{array}\right]
  \left[\begin{array}{cccc}
  0       &0     &\cdots     &0      \\
  0       &0       &X_2       &\vdots \\
  \vdots  &\ddots  &\ddots     &0      \\
  0     &\cdots  &0          &0
  \end{array}\right]\cdots
  \left[\begin{array}{cccc}
  0       &0     &\cdots     &0      \\
  0       &0       &\ddots     &\vdots \\
  \vdots  &\ddots  &\ddots     &X_{n-1}      \\
  0     &\cdots  &0          &0
  \end{array}\right]
    \left[\begin{array}{cccc}
  0       &0     &\cdots     &0      \\
  0       &0       &\ddots     &\vdots \\
  \vdots  &\ddots  &\ddots     &0      \\
  X_n     &\cdots  &0          &0
  \end{array}\right]$$
$$
=\left[\begin{array}{cccc}
  X_1\cdots X_n       &0     &\cdots     &0      \\
  0       &0       &0       &\vdots \\
  \vdots  &\ddots  &\ddots     &0      \\
  0     &\cdots  &0          &0
  \end{array}\right]=
  \left[\begin{array}{cccc}
  A_1X^n       &0     &\cdots     &0      \\
  0       &0       &0       &\vdots \\
  \vdots  &\ddots  &\ddots     &0      \\
  0     &\cdots  &0          &0
  \end{array}\right]\subseteq D.
$$
 Then, with $\left[\begin{smallmatrix}A_1 & 0_{1\times(n-1)}\\ 0_{(n-1)\times1} & 0_{(n-1)\times(n-1)} \end{smallmatrix}\right]\subseteq D$ and $\left[\begin{smallmatrix}A_1X^n & 0_{1\times(n-1)}\\ 0_{(n-1)\times1} & 0_{(n-1)\times(n-1)} \end{smallmatrix}\right]\subseteq D$ we can generate $\w E_{11}=\left[\begin{smallmatrix}E_{11} & 0_{1\times(n-1)}\\ 0_{(n-1)\times1} & 0_{(n-1)\times(n-1)} \end{smallmatrix}\right]\subseteq D$, because $E_{11}=B=A_1[A_1X^n]$. Now, for $k,l=0,\ldots,n-1$ we have that the product $\w X^{-k}\w E_{11}\w X^l\subseteq D$ is the space that, with the matrix notation, has $X_k^*\cdots X_1^*BX_1\cdots X_l=X^{-k}BX^l=E_{1+k\,1+l}$ at the $(l+1,k+1)$-entry and $0$ elsewhere. As $(k+1,l+1)$ ranges over all entries when $k,l=0,\ldots,n-1$, we conclude that $Z^*Z=[E_{ij}]_{i,j=1}^n\subseteq D=\Im\varphi\rtimes\psi$ as desired.
\end{proof}

\begin{remark} For the especial case of Theorem \ref{elteo} in which $A_i=A$ and $X_i=X$ for $i=1,\ldots,n$, we obtain $A^n\rtimes X^n_{\sigma}\sim A\rtimes X^{\otimes n}$. As pointed out in the introduction, this can be viewed as a generalization to the $C^*$-bimodule context of Green's theorem \cite[Theorem 4.22]{W} for the case $G=\Z$ and $H=n\Z$. That is, if $\alpha\colon A\to A$ is a $*$-automorphism, taking $X$ as the trivial Hilbert bimodule $_AA{_A}$ twisted by $\alpha$, the equivalence $A^n\rtimes X^n_{\sigma}\sim A\rtimes X^{\otimes n}$ becomes $C_0(G/H,A)\rtimes G\sim A\rtimes_{\alpha|_H}H$ where $\alpha$ also denotes the action of $\Z$ generated by the automorphism.
\end{remark}

\begin{corollary}\label{coro1} Let $_AX_B$ and $_BY_A$ be full right Hilbert bimodules. Then
  $$A\rtimes(X\otimes Y)\sim B\rtimes(Y\otimes X).$$
\end{corollary}

\begin{proof} The twisted sums $(X\oplus Y)_{\sigma}$ and $(Y\oplus X)_{\sigma}$ are isomorphic bimodules, the pair $(\varphi,\psi)$ where $\varphi\colon A\oplus B\to B\oplus A$, $\varphi(a,b)=(b,a)$, and $\psi\colon X\oplus Y\to Y\oplus X$, $\psi(x,y)=(y,x)$, being an isomorphism. Then, the corresponding crossed products are isomorphic as well. Therefore
$$A\rtimes(X\otimes Y)\sim (A\oplus B)\rtimes (X\oplus Y)_{\sigma}\cong (B\oplus A)\rtimes (Y\oplus X)_{\sigma}\sim B\rtimes(Y\otimes X),$$
where we applied twice Theorem \ref{elteo} for $n=2$.
\end{proof}

\begin{corollary} (\cite[Theorem 4.1]{AEE}) \label{coro2} Let $_AX_A$ and $_BY_B$ be full right Hilbert bimodules and $_AM_B$ an equivalence bimodule such that $X\otimes M\cong M\otimes Y$. Then $$A\rtimes X\sim B\rtimes Y.$$
\end{corollary}
\begin{proof} As $M$ is an equivalence we have $A\cong M\otimes M^*$, where $A$ is considered as the trivial Hilbert $A-A$ bimodule and $M^*$ denotes the conjugated bimodule of $M$. Then $X\cong A\otimes X\cong M\otimes M^*\otimes X$. Besides, $M^*\otimes X\otimes M\cong Y$ by hypothesis. Then, as all these isomorphisms give isomorphic crossed products, we have
 $$A\rtimes X\cong A\rtimes (A\otimes X)\cong A\rtimes (M\otimes M^*\otimes X)\sim B\rtimes (M^*\otimes X\otimes M)\cong B\rtimes Y,$$
 where we applied Corollary \ref{coro1} to commute $M$ and $M^*\otimes X$.
\end{proof}

\begin{remark} In \cite{AA} the augmented Cuntz-Pimsner $C^*$-algebra $\w{\O}_X$ associated to an $A-A$ correspondence $X$ (see \cite{P}) is described as a crossed product $A_{\infty}\rtimes X_{\infty}$, where $X_{\infty}$ is a Hilbert $A_{\infty}-A_{\infty}$ bimodule constructed out of the $A-A$ correspondence $X$. Then, combining this description with \cite[Theorem 4.1]{AEE} (Corollary \ref{coro2} here) it is shown an analogue of this theorem in the context of augmented Cuntz-Pimsner $C^*$-algebras (\cite[Theorem 4.7]{AA}). 

We believe that using similar techniques to those of \cite{AA}, it is possible to obtain versions of Theorem \ref{elteo} and Corollary \ref{coro1} for augmented Cuntz-Pimsner algebras. For example, the corresponding version of the Corollary \ref{coro1} should establish that $\w{\O}_{X\otimes Y}\sim\w{\O}_{Y\otimes X}$ for full correspondences $_AX_B$ and $_BY_A$.
\end{remark}

{\bf Acknowledgement.} The author wishes to thank his friend Janine Bachrachas for her help editing this article.

\bibliographystyle{amsplain}

\end{document}